\documentclass[12pt, parskip=full]{scrartcl}

\usepackage[ngerman, english]{babel}

\usepackage[utf8]{inputenc}

\setkomafont{sectioning}{\normalfont\normalcolor\bfseries}

\usepackage{amsmath}
\usepackage{amsfonts}
\usepackage{amssymb}
\usepackage{amsthm}
\usepackage{stmaryrd}

\theoremstyle{plain}
\newtheorem{theorem}{Theorem}[section]
\newtheorem{lemma}[theorem]{Lemma}

\newtheorem{assumption}{Assumption}

\theoremstyle{definition}
\newtheorem{definition}{Definition}[section]

\theoremstyle{remark}

\newtheorem{remark}[theorem]{Remark}
\newtheorem{example}[theorem]{Example}

\usepackage[round]{natbib}

\usepackage{url}

\usepackage{abstract}

\usepackage{hyperref}
\usepackage{graphicx}
\usepackage{xcolor}

\usepackage{pgf}
\usepackage{tikz}
\usetikzlibrary{arrows,automata,positioning,calc}

\usepackage{changes}

\newcommand\xqed[1]{%
	\leavevmode\unskip\penalty9999 \hbox{}\nobreak\hfill
	\quad\hbox{#1}}
\newcommand\demo{\xqed{$\triangle$}}

\title{A Myopic Adjustment Process for Mean Field Games with Finite State and Action Space}
\author{Berenice Anne Neumann\thanks{University of Trier, Department IV, Universitätsring 19, 54296 Trier, Germany}}

\begin{document}
	
	\maketitle
	\begin{abstract}
		In this paper, we introduce a natural learning rule for mean field games with finite state and action space, the so-called myopic adjustment process. The main motivation for these considerations are the complex computations necessary to determine dynamic mean-field equilibria, which make it seem questionable whether agents are indeed able to play these equilibria. We prove that the myopic adjustment process converges locally towards stationary equilibria with deterministic equilibrium strategies under rather broad conditions. Moreover, for a two-strategy setting, we also obtain a global convergence result under stronger, yet intuitive conditions. 
		 		
 		
	\end{abstract}

	\section{Introduction}

	Mean field games have been introduced by \citet{LasryJapanese2007} and \citet{HuangNCE2006} in order to make dynamic games with a large number of players tractable. The central idea is to approximate these games with many players by a game with a continuum of anonymous players. Thereafter a vibrant field of research emerged in particular concerning games where the dynamics of the individual players are described by diffusions. For a first overview consider the monographs by \citet{BensoussanMFG} and by \citet{CarmonaMFG2018, CarmonaMFGPartTwo2018} or the lecture notes by \citet{CardaliaguetLectureNotes}. However, a central problem for applications is that equilibria are described by forward-backward systems of (stochastic) differential equations and are, therefore, notoriously intractable.
	
 	Recently, also mean field games with finite state space have been considered for example \citet{GomesConti2013, CecchinProbabilistic2018, TrierPreprint, CarmonaFinite, DoncelPaper, NeumannComputation} and \citet[Section 7.2]{CarmonaMFG2018}. 
 	Moreover, several applications have been considered including the spread of corruption, botnet defence, paradigm shift in science and consumer choice \citep{BesancenotParadigm2015,KolokoltsovBotnet2016, KolokoltsovCorruption2017, GomesSocio2014}.
 	In \citet{NeumannComputation} stationary equilibria for mean field games with finite state and action space have been considered and results yielding an explicit characterization of these equilibria have been derived.
 	For applications it is now clearly desirable to understand whether stationary equilibria are an adequate description of agents' behaviour. 
 	One approach is to discuss the question in how far these equilibria are suitable limit objects of dynamic equilibria when the finite time horizon tends to infinity (consider \citet{KolokoltsovBotnetCorruption} for the analysis of an example). 
 	The second approach is to understand in how far stationary equilibria arise when agents apply certain only partly rational decision rules. This approach is a classical one in standard game theory and known as learning.
 	
 	Why players should play equilibrium strategies is a classical concern. 
 	The explanation that agents arrive at these strategies from ``introspection and computation'' is challenged by many facts:
 	The computational complexity of the problems at hand, the question which equilibrium to choose in case of multiple equilibria and experimental evidence \citep{FudenbergLearning}. 
 	Because it was observed in experiments that agents ``learn'' to play equilibria after some time, many authors focussed on the definition and analysis of partially rational ``learning rules'' (for example fictitious play or partial best response) mostly for static games. For an overview consider the monograph by \citet{FudenbergLearning} or the survey by \citet{NachbarLearning}.
 	
 	Recently, learning has also been discussed for mean field games with diffusion-based dynamics:
 	On the one hand, \citet{CardaliaguetLearning} introduced fictitious play for repeated games with finite time horizon which thereafter has been analysed in subsequent publications \citep{HadikhanlooPaper, BrianiLearningMFG} and extended to discrete time finite state mean field games with finite time horizon \citep{HadikhanlooPhDLearning}.
 	On the other hand, \citet{MouzouniMyopic} introduced a learning procedure similar to the myopic adjustment process considered here. He proved existence and local convergence under strong assumptions (quadratic Hamiltonian and Lasry-Lions monotonicity condition).
 	We highlight, that the methods for diffusion-based mean field games cannot be adapted to our setting of a mean field game with finite state and action space, since the crucial assumption for most methods (a unique optimizer of the Hamiltonian) is typically not satisfied (see \citet[Remark 2.5]{NeumannComputation}).

	This is the first paper that discusses learning in mean field games with finite state and action space. We will introduce a myopic adjustment process, where agents choose to play a best response for the scenario that the current population distribution will persist for all future times. This definition yields to a formulation of this process as a differential inclusion and we can prove existence of this process under a continuity assumption.
	
	Then we discuss under which conditions we can expect local convergence towards stationary equilibria. 
	First, we explain that a general convergence result is only possible for deterministic stationary equilibria. 
	Thereafter, we obtain a first convergence result for dynamics constant in $m$.
	Namely, under a classical irreducibility condition we establish local convergence towards stationary equilibria with a deterministic equilibrium strategy. 
	For general dynamics, under some technical conditions as well as a condition similar to the irreducibility condition in the constant case, we are again able to prove convergence towards stationary equilibria with a deterministic equilibrium strategy.	
	We conclude the analysis of local convergence with the discussion of examples.
	
	In the end of the paper, we turn to the question of global convergence for games with only two sensible deterministic stationary strategies. 
	Under certain technical conditions, we obtain that whenever the instantaneous change of the population is non-orthogonal to the set of points where both strategies are simultaneously optimal, global convergence towards deterministic stationary strategies happens or the trajectories stay in the set where both strategies are simultaneously optimal. 
	In examples this is often enough to prove global convergence towards stationary equilibria in mixed strategies, which we illustrate in an example.
	
	The rest of the paper is structured as follows: Section \ref{Section:Model} describes the mean field game model considered in this paper. Section \ref{Section:Process} then introduces the myopic adjustment process and justifies its definition as a sensible partially rational learning rule. Moreover, it presents the myopic adjustment process for a simple example. In Section \ref{Section:Local} we study the local convergence of the myopic adjustment process and in Section \ref{Section:Global} we then study the global convergence for the special case of two strategies.
	
	\section{Stationary Equilibria of Mean Field Games with Finite State and Action Space}
	\label{Section:Model}

	This section describes the mean field game model.
	The setup is the same as in \citet{NeumannComputation} and we refer the reader to this paper for details regarding well-definition and intuitions.
	Moreover, we remark that the model has been first introduced in an analytic formulation and without the notion of stationary equilibria in \citet{DoncelPaper}.
	
	Let $\mathcal{S}=\{1, \ldots, S\}$ ($S>1$) be the set of possible states of each player and let $\mathcal{A}=\{1, \ldots, A\}$ be the set of possible actions. 
	With $\mathcal{P}(\mathcal{S})$ we denote the probability simplex over $\mathcal{S}$ and with $\mathcal{P}(\mathcal{A})$ the probability simplex over $\mathcal{A}$. 
	A \textit{(mixed) strategy} is a measurable function $\pi: \mathcal{S} \times [0,\infty) \rightarrow \mathcal{P}(\mathcal{A})$, $(i,t) \mapsto (\pi_{ia}(t))_{a \in \mathcal{A}}$ with the interpretation that $\pi_{ia}(t)$ is the probability that at time $t$ and in state $i$ the player chooses action $a$. 
	A strategy $\pi=d:\mathcal{S} \times [0,\infty) \rightarrow \mathcal{P}(\mathcal{A})$ is \textit{deterministic} if it satisfies for all $t \ge 0$ and for all $i \in \mathcal{S}$ that there is an $a \in \mathcal{A}$ such that $d_{ia}(t)=1$ and $d_{ia'}=0$ for all $a' \in \mathcal{A} \setminus \{a\}$. 
	Sometimes the following equivalent representation is helpful: Namely, we represent a deterministic strategy as a function $d: \mathcal{S} \times [0,\infty) \rightarrow \mathcal{A}, (i,t) \mapsto d_i(t)$ with the interpretation that $d_i(t)=a$ states that at time $t$ in state $i$ action $a$ is chosen.
	A \textit{stationary strategy} is a map $\pi: \mathcal{S} \times [0,\infty) \rightarrow \mathcal{P}(\mathcal{A})$ such that $\pi_{ia}(t) = \pi_{ia}$ for all $t \ge 0$. 
	With $\Pi$ we denote the set of all (mixed) strategies and with $\Pi^s$ the set of all stationary strategies.
	Similarly, we denote by $D$ the set of all deterministic strategies and by $D^s$ the set of all deterministic stationary strategies.
	
	Let for all $a \in \mathcal{A}$ and $m \in \mathcal{P}(\mathcal{S})$ the matrices $(Q_{\cdot \cdot a}(m))_{a \in \mathcal{A}}$ be conservative generators, that is $Q_{ija}(m) \ge 0$ for all $i, j \in \mathcal{S}$ with $i \neq j$ and $\sum_{j \in \mathcal{S}} Q_{ija}(m) =0$ for all $i \in \mathcal{S}$. 
	The individual dynamics of each player given a Lipschitz continuous flow of population distributions $m: [0, \infty) \rightarrow \mathcal{P}(\mathcal{S})$ and a strategy $\pi: \mathcal{S} \times [0, \infty) \rightarrow \mathcal{P}(\mathcal{A})$ are given as a Markov process $X^\pi(m)$ with given initial distribution $x_0 \in \mathcal{P}(\mathcal{S})$ and infinitesimal generator given by the $Q(t)$-matrix $$\left( Q^\pi(m(t),t) \right)_{ij} = \sum_{a \in \mathcal{A}} Q_{ija}(m(t)) \pi_{ia}(t).$$  	
	Given the initial condition $x_0 \in \mathcal{P}(\mathcal{S})$, the goal of each player is to maximize his expected discounted reward, which is given by 
	\begin{equation}
	\label{valueFunction}
	V_{x_0}(\pi, m) = \mathbb{E} \left[ \int_0^\infty \left( \sum_{a \in \mathcal{A}} r_{X^\pi(m) a} (m(t)) \pi_{X^\pi(m)a}(t) \right) e^{-\beta t} \text{d}t \right],
	\end{equation} where $r: \mathcal{S} \times \mathcal{A} \times \mathcal{P}(\mathcal{S}) \rightarrow \mathbb{R}$ is a real-valued function and $\beta \in (0,1)$ is the discount factor. 
	That is, for a fixed flow of population distributions $m: [0,\infty) \mapsto \mathcal{P}(\mathcal{S})$ the individual agent's decision problem is a Markov decision process with expected discounted reward criterion and time-inhomogeneous reward functions and transition rates. 
	
	In this paper we work under the following standing assumption, which ensures the well-definition of the model as well as the existence of dynamic as well as stationary equilibria \citep{NeumannComputation}:
	
	\begin{assumption}
		\label{assumption:continuous}
		For all $i, j \in \mathcal{S}$ and all $a \in \mathcal{A}$ the function $m \mapsto Q_{ija}(m)$ mapping from $\mathcal{P}(\mathcal{S})$ to $\mathbb{R}$ is Lipschitz-continuous in $m$ . 
		For all $i \in \mathcal{S}$ and all $a \in \mathcal{A}$ the function $m \mapsto r_{ia}(m)$ mapping from $\mathcal{P}(\mathcal{S})$ to $\mathbb{R}$ is continuous in $m$.
	\end{assumption}
	
	\begin{definition}
		\label{definition:Dynamic}
		Given an initial distribution $m_0 \in \mathcal{P}(\mathcal{S})$, a \textit{mean field equilibrium} is a pair $(m,\pi)$ consisting of a flow of population distributions $m: [0,\infty) \rightarrow \mathcal{P}(\mathcal{S})$ with $m(0)=m_0$ and a strategy $\pi: \mathcal{S} \times [0,\infty) \rightarrow \mathcal{P}(\mathcal{A})$ such that
		\begin{itemize}
			\item the distribution of the process $X^\pi(m)$ at time $t$ is given by $m(t)$, and
			\item $V {_{m_0}}(\pi, m) \ge V {_{m_0}}(\pi',m)$ for all $\pi' \in \Pi$. 
		\end{itemize}
	\end{definition} 
	
	\begin{definition}
		\label{definition:Stationary}
		A stationary mean field equilibrium is given by a stationary strategy $\pi$ and a vector $m \in \mathcal{P}(\mathcal{S})$ such that
		\begin{itemize}
			\item the law of $X^\pi(m)$ at any point in time $t$ is given by $m$, and
			\item  {for any initial distribution $x_0 \in \mathcal{P}(\mathcal{S})$ we have} $V {_{x_0}}(\pi, m) \ge V {_{x_0}}(\pi', m)$ for all $\pi' \in \Pi$.
		\end{itemize}
	\end{definition}

	\section{The Myopic Adjustment Process}
	\label{Section:Process}
	
	In general, it is not possible to compute dynamic mean field equilibria for the considered game; it is not even possible to explicitly characterize solutions of the individual control problem for a given non-constant flow of population distributions.
	Moreover, also in the case of a finite time horizon, the search for equilibria can only be reduced to a forward-backward system of ODEs, which can, most of the time, be only solved numerically (see \citet{TrierPreprint}).
	The aim of this section is to motivate and define a reasonable alternative decision mechanism for the agents.
	
	In contrast to \citet{CardaliaguetLearning} we cannot assume that the game is played repeatedly, but instead we have to assume that the agent changes his strategy during the game.
	For this we note that the game at time $t$ with current distribution $m$ is, due to the time-homogeneous formulation and the infinite time horizon, equivalent to the game started at time $0$ with initial distribution $m$.
	Moreover, we remind ourselves that the influence of the individual agent on the game characteristics and thus on the payoff of the other players is negligible.
	Therefore, it is reasonable to assume that the agents do not try to influence the other players' choices, but that they only maximize their own payoff.
	Because of time-homogeneity and the negligible influence on other players we assume that the agents choose Markovian strategies that only depend on the current state and current population distribution.
	
	We assume that the agent when choosing an optimal strategy given the current population distribution $m$ assumes that the population distribution is constant. This myopia is not only a classical simplification, but the only sensible prediction an agent can compute in this setting.
	Indeed, in general one cannot compute optimal strategies for the Markov decision processes with non-stationary transition rates and rewards that arise in this setting.
	Given such a constant prediction of the behaviour of the population, the optimization problem becomes a tractable Markov decision process with stationary transition rates and rewards. 
	It is well known that there is always an optimal stationary strategy for the considered optimization problem \citep{GuoCTMDP2009} and it is again natural to assume that agents choose such a stationary strategy.

	We remark that this assumption that agents choose a stationary strategy is classical and that there are several conceptual reasons for the use of these strategies (see \citet{MaskinMPE2001}).
	Namely, stationary Markov strategies are the simplest (rational) form of decision-making in this context.
	Moreover, this type of strategies is related to subgame perfection, which is as discussed earlier a reasonable requirement in our setting.
	Indeed, these strategies ensure that a game with the same relevant characteristics (i.e. current state and population distribution) is played in the same way.
	Finally, the restriction on this type of strategies reduces the number of possible best responses and thus increase predictive power.

	For our purpose now the following result explicitly characterizing the set of all optimal stationary strategies given a constant flow of population distributions proves to be useful (see \citet[Section 3]{NeumannComputation}): Let us denote by $V^\ast_j(m)$ the unique solution of the optimality equation for the Markov decision process and define 
	\[
		O_i(m)=: \text{argmax}_{a \in \mathcal{A}} \left\{ r_{ia}(m) + \sum_{j \in \mathcal{S}} Q_{ija}(m) V_j^\ast(m)\right\}.
	\]
	Furthermore, we set
	\[
		\mathcal{D}(m) := \{ d: \mathcal{S} \rightarrow \mathcal{A} |d(i) \in O_i(m) \quad \text{for all} \quad i \in \mathcal{S} \}.
	\]
	Then the set of all optimal stationary strategies is given by $\text{conv}(\mathcal{D}(m))$.
	
	We argued that every agent will choose a strategy from the set  $\text{conv}(\mathcal{D}(m))$.
	However, we cannot assume that all agents choose a particular strategy nor that the agents or groups of them agree on a common strategy. Moreover, we also cannot describe which agents will choose which strategy. 
	The only sensible assumption is that the population chooses aggregately a strategy from the set  $\text{conv}(\mathcal{D}(m))$. 	
	The next lemma describes how the population's distribution evolves if all agents adopt this decision mechanism:
	
	\begin{lemma}
		Let $m:[0,\infty) \rightarrow \mathcal{P}(\mathcal{S})$ be the distribution of the population, where at time $t \ge 0$ any agent chooses a strategy from $\text{conv}(\mathcal{D}(m))$. Then
		\begin{equation}
		\label{Equation:InstantaneousChanges}
		\dot{m}(t) \in F(m(t)) := \text{conv} \left\{ \left( \sum_{i \in \mathcal{S}} \sum_{a \in \mathcal{A}} m_i Q_{ija}(m(t)) d_{ia} \right)_{j \in \mathcal{S}}: d \in \mathcal{D}(m(t)) \right\}
		\end{equation} for almost all $t \ge 0$.
	\end{lemma}
	
	\begin{proof}
		Using the Kolmogorov forward equation, the individual agent's dynamics given any strategy $\pi \in \Pi^s$ can be equivalently described as being the solution of the ordinary differential equation (in the sense of Caratheodory) 
		$$\dot{x}_j(t) = \sum_{i \in \mathcal{S}} x_i(t) \sum_{a \in \mathcal{A}} Q_{ija}(m(t)) \pi_{ia}(t)$$ with initial condition $x(0)=x_0$.
		Since the aggregated strategy $\pi$ of the population satisfies $\pi \in \text{conv} (\mathcal{D}(m))$ the desired claim follows.
	\end{proof}

	With these preparations we define the myopic adjustment process as a solution of the differential inclusion in the sense of \citet{DeimlingMultivaued} given by \eqref{Equation:InstantaneousChanges}. 	
	Namely, a trajectory of the myopic adjustment process is an absolutely continuous function $m: [0, \infty) \rightarrow \mathcal{P}(\mathcal{S})$ such that 
	\begin{equation}
		\label{Equation:DifferentialInclusion}
		\dot{m}(t) \in F(m(t)) \quad \text{for almost all} \quad t \ge 0, \quad m(0) = m_0.
	\end{equation}
	We remark, that the use of differential inclusion as a modelling tool for situations where uncertainty, the absence of control or a variety of available dynamics occurs is classical \citep{AubinDifferentialInclusions}.

	Before we start the analysis of the long-term behaviour, we note, relying a classical existence result for differential inclusions, that under Assumption \ref{assumption:continuous} always a trajectory of the differential inclusion exists. The proof can be found in the appendix \ref{appendix}.
	
	\begin{theorem}
		\label{Theorem:Existence}
		The differential inclusion defined by \eqref{Equation:InstantaneousChanges} and \eqref{Equation:DifferentialInclusion} admits a solution $m: [0,\infty) \rightarrow \mathcal{P}(\mathcal{S})$.
	\end{theorem}
	
	For our purpose, it is central to understand how stationary equilibria and the trajectories of the myopic adjustment process interact. The following observation, which is classical for many learning procedures, is a first step:
	
	\begin{remark}
		By definition, a point is a stationary point of \eqref{Equation:DifferentialInclusion} if and only if it is a stationary mean field equilibrium. \demo
	\end{remark}

	In the Section \ref{Section:Local} we will analyse, whether the myopic adjustment process started close to a stationary equilibrium converges towards it. Thereafter, in Section \ref{Section:Global} we analyse in the setting of two sensible deterministic stationary strategies under which conditions convergence towards some stationary equilibrium irrespective of the starting point can be expected.

	We conclude this section by discussing the shape of the myopic adjustment process for an example linked to consumer choice in the mobile phone sector, for which previously in \citet{NeumannComputation} the stationary equilibria have been computed:
	
	\begin{example}
		\label{Example:Definition}
		The agents can choose between two providers and their utility is increasing in the share of customers using the same provider. The agents can switch the provider facing a time-unit cost $c$. However, the decision is not implemented immediately, but according to a Poisson process with rate $b$. The formal description of the model is given as follows: Let $\mathcal{S}=\{1,2\}$ and $\mathcal{A}= \{\mathit{stay}, \mathit{change}\}$. Let $\delta>0$ be small and define
		\[
		f_\delta: \mathbb{R} \rightarrow \mathbb{R} \quad y \mapsto \begin{cases}
		\frac{1}{2\delta} y^2 + \frac{\delta}{2} &\text{if } y \le \delta \\
		y &\text{if } y > \delta
		\end{cases}.
		\]
		Then the transition rates and rewards are giben by
		\begin{align*}
		Q^{\mathit{change}}(m) &= \begin{pmatrix}
		-b & b \\ b & -b
		\end{pmatrix} \\
		Q^{\mathit{stay}}(m) &= \begin{pmatrix}
		- \epsilon & \epsilon \\ \epsilon & -\epsilon
		\end{pmatrix} \\
		r^{\mathit{change}}(m) &= \begin{pmatrix}
		\ln (f_\delta(m_1)) + s_1 -c \\
		\ln(f_\delta(m_2)) + s_2 -c 
		\end{pmatrix} &
		r^{\mathit{stay}}(m) &= \begin{pmatrix}
		\ln (f_\delta(m_1)) + s_1 \\
		\ln(f_\delta(m_2)) + s_2  
		\end{pmatrix},
		\end{align*} where $\epsilon$, $b$, $s_1$, $s_2$ and $c$ are positive constants with $\epsilon<b$.
		
		Let us write $(a_1,a_2)$ for the deterministic strategy $d$ such that $d(i)=a_i$. 
		In \citet{NeumannComputation} it is then shown that
		\[
		\mathcal{D}(m) = \begin{cases}
		(\mathit{change}, \mathit{stay}) &\text{if } m_1 < k_1 \\
		\text{conv}\{ (\mathit{change},\mathit{stay}), (\mathit{stay},\mathit{stay})\} &\text{if } m_1 = k_1 \\
		(\mathit{stay},\mathit{stay}) &\text{if } k_1 < m_1 < k_2 \\
		\text{conv} \{ (\mathit{stay}, \mathit{stay}), (\mathit{stay}, \mathit{change})\} &\text{if } m_1 = k_2 \\
		(\mathit{stay}, \mathit{change}) &\text{if } m_1 > k_2
		\end{cases},
		\] where 
		\[
		k_1 = \frac{\text{exp}\left( -\frac{c(\beta + 2 \epsilon)}{b-\epsilon}-s_1+s_2 \right)}{1+ \text{exp}\left( -\frac{c(\beta + 2 \epsilon)}{b-\epsilon}-s_1+s_2 \right)}
		\quad \text{and} \quad 
		k_2 = \frac{\text{exp}\left( \frac{c(\beta + 2 \epsilon)}{b-\epsilon}-s_1+s_2 \right)}{1+ \text{exp}\left( \frac{c(\beta + 2 \epsilon)}{b-\epsilon}-s_1+s_2 \right)}.
		\]
		Thus, the myopic adjustment process is defined as a solution of the differential inclusion $\dot{m}(t) \in F(m(t))$ with
		\[
		F(m) = \begin{cases}
		\begin{pmatrix}
		- bm_1(t) + \epsilon m_2(t) \\
		b m_1(t) - \epsilon m_2(t)
		\end{pmatrix}
		&\text{if } m_1(t) < k_1 \\
		\text{conv} \left\{\begin{pmatrix}
		- bm_1(t) + \epsilon m_2(t) \\
		b m_1(t) - \epsilon m_2(t)
		\end{pmatrix},\begin{pmatrix}
		- \epsilon m_1(t) + \epsilon m_2(t) \\
		\epsilon m_1(t) - \epsilon m_2(t)
		\end{pmatrix} \right\} &\text{if } m_1(t) = k_1 \\
		\begin{pmatrix}
		- \epsilon m_1(t) + \epsilon m_2(t) \\
		\epsilon m_1(t) - \epsilon m_2(t)
		\end{pmatrix} &\text{if } k_1< m_1(t) < k_2 \\
		\text{conv} \left\{\begin{pmatrix}
		- \epsilon m_1(t) + \epsilon m_2(t) \\
		\epsilon m_1(t) - \epsilon m_2(t)
		\end{pmatrix}, \begin{pmatrix}
		- \epsilon m_1(t) + b m_2(t) \\
		\epsilon m_1(t) - b m_2(t)
		\end{pmatrix} \right\} &\text{if } m_1(t) = k_2 \\
		\begin{pmatrix}
		- \epsilon m_1(t) + b m_2(t) \\
		\epsilon m_1(t) - b m_2(t)
		\end{pmatrix} &\text{if } m_1(t) > k_2
		\end{cases} .
		\]
		In Figure \ref{ConsumerChoiceIllustration} we illustrate the behaviour of this process for a parameter choice that yields for any initial condition to a unique solution of the differential inclusion.
	\end{example}  	
	
	\begin{figure}[h]
		\begin{center}
			\includegraphics[scale=0.6]{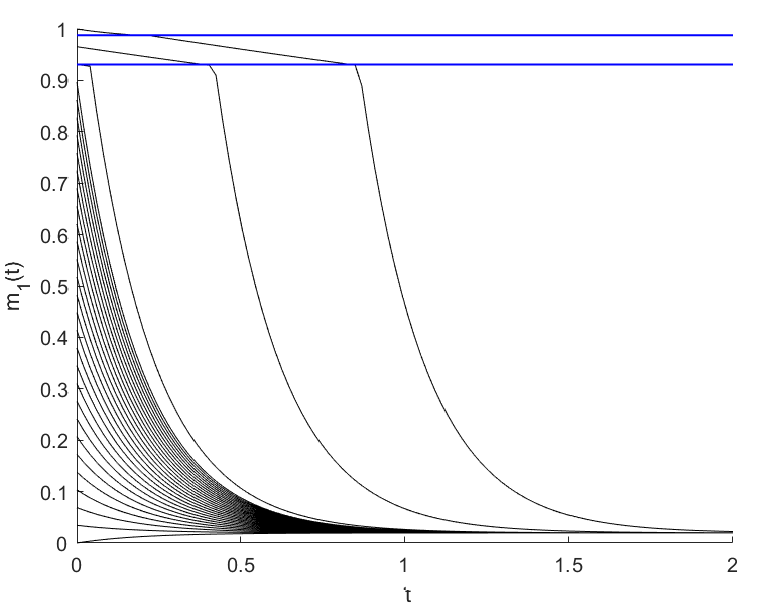}
		\end{center}
	\caption{Illustration of Example \ref{Example:Definition}:
				The figure shows $m_1(t)$ for several trajectories of the myopic adjustment process given different initial conditions. The blue vertical lines represent the thresholds at which the set of optimal strategies changes.}
		\label{ConsumerChoiceIllustration}
	\end{figure}

	\section{Local Convergence}
	\label{Section:Local}
	
	This section discusses the question of local convergence, that is we analyse under which conditions we can expect that for an initial condition close to a stationary equilibrium the trajectories of \eqref{Equation:DifferentialInclusion} converge towards that equilibrium.
	
	Local convergence of the myopic adjustment process will in general not happen if the equilibrium distribution $\bar{m}$ satisfies $|\mathcal{D}(\bar{m})|>1$.
	Indeed, in this case there are points arbitrary close to the equilibrium at which a different strategy than the equilibrium strategy is optimal and for which we cannot tell anything about the behaviour of the trajectories given this strategy -- it might happen that we are pushed away from the equilibrium. 
	Thus, it remains to verify whether we have local convergence towards a stationary deterministic mean field equilibrium $(\bar{m},d)$ where $d$ is the unique optimal strategy at $\bar{m}$. 
	As a first step, we observe the following:
	
	\begin{lemma}
		Let $d$ be the unique optimal stationary strategy for $\bar{m}$ (that is $\mathcal{D}(\bar{m})= \{d\}$). Then there is an $\epsilon>0$ such that for all $m' \in N_\epsilon(\bar{m})$ we have $\mathcal{D}(m') = \{d\}$. 
	\end{lemma}

	\begin{proof}
		Let $V^d(m)$ denote the expected discounted reward if the player chooses strategy $d$ and the population distribution is $m$ for all times. In \citet[Section 3]{NeumannComputation} it is shown that $m \mapsto V^d(m)$ is continuous. Since $\mathcal{D}(\bar{m})= \{d\}$, we have $V^d(\bar{m}) > V^{\hat{d}}(\bar{m})$ pointwise for all $\hat{d} \in D^s \setminus \{d\}$. By continuity there is an $\epsilon>0$ such that $V^d(m') > V^{\hat{d}}(m')$ pointwise for all $m' \in N_\epsilon(\bar{m})$ and $\hat{d} \in D^s\setminus \{d\}$. Thus, $\mathcal{D}(m') = \{d\}$ for all $m' \in N_\epsilon(\bar{m})$.
	\end{proof}
	
	This lemma now yields that for all $m' \in N_\epsilon(\bar{m})$ we have $F(m') = \{(Q^d(m'))^T m'\}$. 
	Thus, it suffices to investigate whether there is a $\delta>0$ such that for all $m_0 \in \mathcal{P}(\mathcal{S})$ satisfying $|m_0-\bar{m}|< \delta$ we have that the solution of $\dot{m}(t) = (Q^d(m(t)))^Tm(t)$ lies in $N_\epsilon(\bar{m})$ for all $t \ge 0$ and converges towards $\bar{m}$. 
	This question is closely linked to the notion of asymptotically stable solutions of autonomous ordinary differential equations $\dot{x}=f(x)$.
	However, we do not consider arbitrary initial conditions in $N_\delta(\bar{x})$, but only  those that lie in $\mathcal{P}(\mathcal{S})$. Moreover, we always face, since $Q(m)$ is conservative for all $m \in \mathcal{P}(\mathcal{S})$, a zero eigenvalue of the Jacobian.
	Thus, general stability results are not applicable.
	
	The first positive result we present covers the case where the dynamics given the equilibrium strategy are constant, that is $\dot{m} = (Q^d)^T m$. In this setting it is even possible to explicitly characterize what ``close'' means:
	
	\begin{theorem}
		\label{Theorem:LocalConstant}
		Let $(\bar{m}, d)$ be a stationary mean field equilibrium such that $\mathcal{D}(\bar{m})=\{d\}$ (that is, $d$ is the unique optimal strategy at $\bar{m}$). Furthermore, assume that $Q^d(m)$ is constant in $m$ and an irreducible generator.
		Then  there is a $\delta>0$ such that any solution of the myopic adjustment process \eqref{Equation:DifferentialInclusion} with initial condition $m_0 \in \mathcal{P}(\mathcal{S}) \cap N_\delta(\bar{m})$ converges exponentially fast to $\bar{m}$, i.e. there are constants $C_1,C_2 >0$ such that
		$$||m(t)-\bar{m}|| \le C_1 e^{-C_2t} \quad \text{for all } t \ge 0.$$
	\end{theorem}

	\begin{remark}
		\label{Remark:ExplicitDelta}
		We can explicitly determine $\delta>0$: Namely, let $\lambda_1, \ldots, \lambda_n$ be the eigenvalues of $Q^d$. By \citet[Corollary 4.9]{AsmussenQueues2003} there is an eigenvalue $\lambda_i=0$ with multiplicity one and eigenvector $\bar{m}$ and such that $\text{Re}(\lambda_j)<0$ for all $j \neq i$. Without loss of generality assume that $i=1$ and denote by $m(\lambda_j)$ the multiplicity of $\lambda_j$ and by $(v_j^0, \ldots, v_j^{m(\lambda_j)-1})$ the basis of the generalized eigenspace $\text{Eig}(\lambda_j) := \ker \left( (Q^d)^T- \lambda_jI\right)^{m(\lambda_j)}$.
		Choose $\epsilon>0$ such that for all $\tilde{m} \in N_\epsilon(\bar{m})$ it holds that $\mathcal{D}(\tilde{m}) = \{d\}$ and define
		$$C_j^k:= \sum_{l=0}^{m(\lambda_j)-1} e^{-l} \frac{l^l}{l! (-\text{Re}(\lambda_j))^l} \left|\left|\left((Q^d)^T - \lambda_j I\right)^l v_j^l\right| \right|$$
		for all $j \in \{2, \ldots, n\}$ and $k \in \{0, \ldots, m(\lambda_j)-1\}$, then
		$$\delta := \frac{\epsilon/2 \cdot \min_{j \in \{2, \ldots, n\}, k \in \{0, \ldots, m(\lambda_j)\}} ||v_i^k||}{\min_{j \in \{2, \ldots, n\}, k \in \{0, \ldots, m(\lambda_j)\}} C_i^k}$$ is a suitable choice.
	\end{remark}

	The proof relies on the classical result that the solutions of a linear ODE with initial condition $v \in Eig(\lambda)$ for some eigenvalue $\lambda$  is given by  $x(t) = e^{\lambda t} \sum_{l=0}^{m(\lambda)-1} t^l/l! \left( (Q^d)^T - \lambda I\right)^l v$ \citep[Theorem 2.11]{LogemannODE} as well as a sensible treatment of the eigenvalue structure of $Q^d$:
	
	\begin{proof}
		Let us use the same notation as in Remark \ref{Remark:ExplicitDelta}. 
		We note that $$(\bar{m}, v_2^0, \ldots, v_2^{m(\lambda_2)-1}, \ldots, v_n^0, \ldots, v_i^{m(\lambda_n)-1})$$ is a basis of $\mathbb{R}^S$. 
		Thus, for any initial condition $m_0 \in \mathbb{R}^S$ we find a unique set of constants $(\alpha_i^k)$ such that $m_0 = \alpha_1^0 \bar{m} + \sum_{i=2}^n \sum_{k=0}^{m(\lambda_i)-1} \alpha_i^k v_i^k$.
		By \citet[Theorem 2.11]{LogemannODE} and since $\bar{m}$ is the eigenvector for the eigenvalue $0$ we obtain
		$$m(t) = \alpha_1^0 \bar{m} + \sum_{i=2}^n e^{\lambda_i t} \sum_{k=0}^{m(\lambda_i)-1} \frac{t^l}{l!} ((Q^d)^T-\lambda_i I)^l v_i^k.$$
		Since the continuous time Markov chain with generator $Q^d$ is ergodic, we have that $m(t) \rightarrow \bar{m}$ for $t \rightarrow \infty$. 
		Since moreover $\text{Re}(\lambda_i)<0$ it holds that $\alpha_1^0=1$.
		Thus, using that the function $t \mapsto e^{Re(\lambda_i)t} \frac{t^l}{l!}$ has a unique global maximum in $[0,\infty)$ at $t=- \frac{l}{Re(\lambda_i)}$, we obtain
		\begin{align}
		\notag
			||m(t)- \bar{m}|| &= \left| \left| \sum_{i=2}^n e^{\lambda_i t} \sum_{k=0}^{m(\lambda_i)-1} \alpha_i^k \sum_{l=0}^{m(\lambda_i)-1} \left( (Q^d)^T - \lambda_i \right)^l v_i^k \right| \right| \\
			\label{Equation:YieldsExponentialConvergence}
			&\le \sum_{i=2}^n \sum_{k=0}^{m(\lambda_i)-1} \alpha_i^k \left(\sum_{l=0}^{m(\lambda_i)-1} e^{\text{Re}(\lambda_i)t} \frac{t^l}{l!} \left| \left| \left( (Q^d)^T- \lambda_i I \right)^l v_i^k \right| \right| \right) \\ \notag
			&\le \sum_{i=2}^n \sum_{k=0}^{m(\lambda_i)-1} |\alpha_i^k| C_i^k.
		\end{align}
		If $m_0 \in N_\delta(\bar{m}) \cap \mathcal{P}(\mathcal{S})$, then 
		$$\left| \left| m_0 - \bar{m} \right| \right| = \sum_{i=2}^n \sum_{k=0}^{m(\lambda_i)-1} |\alpha_i^k| \left| \left| v_i^k\right| \right| < \frac{\epsilon \min_{i,k} ||v_i^k||}{\max_{i,k} C_i^k},$$ which implies $\sum_{i=2}^n \sum_{k=0}^{m(\lambda_i)-1} |\alpha_i^k| C_i^k < \epsilon$. 
		Therefore, $m(t) \in N_\epsilon(\bar{m})$ for all $t \ge 0$. The exponential convergence then follows from \eqref{Equation:YieldsExponentialConvergence}.
	\end{proof}

	Also in the case of general dynamics we can provide a similar positive statement. However, we have to impose additional conditions, since the necessary eigenvalue structure does not follow immediately in this setting:
	
	\begin{theorem}
		Let $(\bar{m},d)$ be a stationary mean field equilibrium such that $\mathcal{D}(\bar{m})=\{d\}$ (that is, $d$ is the unique optimal strategy at $\bar{m}$). 
		Let $O \supseteq \mathcal{P}(\mathcal{S})$ be an open set such that $Q^d: O \rightarrow \mathbb{R}^{S \times S}$ is componentwise Lipschitz continuous, the matrix $Q^d(m)$ is a transition rate matrix for all $m \in \mathcal{P}(\mathcal{S})$ and the function $f^d: O \rightarrow \mathbb{R}^S, m \mapsto (Q^d(m))^Tm$ is continuously differentiable in $m$.
		Assume further that the Jacobian $\frac{\partial}{\partial m} f^d(m)$ has a zero eigenvalue with eigenvector $\bar{m}$ and all other eigenvalues have strictly negative real parts. 
		Then there is a $\delta>0$ such that any solution of the myopic adjustment process \eqref{Equation:DifferentialInclusion} with initial condition $m_0 \in \mathcal{P}(\mathcal{S}) \cap N_\delta(\bar{m})$ converges exponentially fast to $\bar{m}$, i.e. there are constants $C_1,C_2 >0$ such that
		$$||m(t)-\bar{m}|| \le C_1 e^{-C_2 t} \quad \text{for all } t \ge 0.$$
	\end{theorem}

	\begin{remark}
		This result covers Theorem \ref{Theorem:LocalConstant}, however, the proof is non-constructive. In particular, we cannot, in contrast to the setting of Theorem \ref{Theorem:LocalConstant}, explicitly describe $\delta$.
	\end{remark}

	The central idea of the proof has also been used in the analysis of nonlinear sinks in \citet{HirschNonlinearSink}, namely to bound $\frac{\partial}{\partial t} ||x(t)||_B$ with $x(t) = m(t)- \bar{m}$ for a suitable basis (and corresponding scalar product (i.e. $\langle b_i, b_j\rangle = \delta_{ij}$) and norm).
	More precisely, since $\langle x,y\rangle_B = x^TCy$ for any basis $B$ we obtain a product rule ($\frac{\partial}{\partial t} \langle g(t), h(t)\rangle_B = \left\langle \frac{\partial}{\partial t} g(t), h(t) \right\rangle_B+ \left\langle g(t), \frac{\partial}{\partial}h(t) \right\rangle_B$) and using this
	$$\frac{\partial}{\partial t} ||x(t)||_B = \frac{\partial}{\partial} \sqrt{\langle x(t),x(t) \rangle_B} = \frac{1}{||x(t)||_B} \langle \dot{x}(t), x(t) \rangle_B.$$
	Also here the central step of the proof is to find a suitable constant $C>0$ such that $$\frac{1}{||x(t)||_B} \langle \dot{x}(t), x(t) \rangle_B \le - C ||x(t)||_B.$$
	However, it is not possible to directly use the techniques of \citet{HirschNonlinearSink} because of the zero eigenvalue (which always occurs since $Q^d(m)$ is conservative).  
	
	\begin{proof}
		Give $\mathbb{R}^S$ new coordinates via the transformation $x=m-\bar{m}$, which in particular means that we now consider $\tilde{f}^d(x)=f^d(x+\bar{m})$. 
		Denote by $A$ the Jacobian matrix of $\tilde{f}^d$ at $0$.
		Let $\lambda_1, \ldots, \lambda_n$ be the eigenvalues of $A$.
		Without loss of generality $\lambda_1=0$.
		By assumption $\text{Re}(\lambda_i)<0$ for all $i \in \{2, \ldots, n\}$, moreover, there are constants $b,c>0$ such that $\text{Re}(\lambda_i)<-b<-c$ for all $i \in \{2, \ldots, n\}$.
		
		By \citet[Chapter 7]{HirschNonlinearSink} there is a basis $$B= (b_1^0, b_2^0, \ldots, b_2^{m(\lambda_1)-1}, \ldots, b_n^0, \ldots, b_n^{m(\lambda_n)-1})$$ with $b_1^0 = \bar{m}$ and corresponding inner product given by $\langle b_i^k, b_j^l\rangle = 1_{\{i=j, k=l\}}$ such that
		\begin{itemize}
			\item[(i)] for all $i \in \{1, \ldots, n\}$ the family $(b_i^0, \ldots, b_i^{m(\lambda_i)-1})$ is a basis of the generalized eigenspace $\text{Eig}(\lambda_i)$
			\item[(ii)] for all $i \in \{2, \ldots, n\}$ and $x \in \text{span}(b_i^0, \ldots, b_i^{m(\lambda_i)-1})$ we have
			$\langle Ax, x \rangle_B \le -b ||x||_B^2.$
		\end{itemize} 
		
		We first obtain that for all vectors $x= \sum_{i=2}^n \sum_{k=0}^{m(\lambda_i)-1} \alpha_i^k b_i^k$ for some constants $(\alpha_i^k)$ we have $\langle Ax,x\rangle_B \le -b ||x||_B^2$. 
		Indeed, since $Ab_i^j \in \text{Eig}(\lambda_i)$ we obtain
		\allowdisplaybreaks
		\begin{align*}
		\langle Ax,x \rangle_B &= \left\langle A \cdot \left( \sum_{i=2}^n \sum_{k=0}^{m(\lambda_i)-1} \alpha_i^k b_i^k \right), \sum_{i=2}^n \sum_{k=0}^{m(\lambda_i)-1} \alpha_i^k b_i^k \right\rangle_B \\
		&= \sum_{i=2}^n \left\langle A \left( \sum_{k=0}^{m(\lambda_i)-1} \alpha_i^k b_i^k \right), \sum_{k=0}^{m(\lambda_i)-1} \alpha_i^k b_i^k \right\rangle_B \\
		&\quad  + \sum_{i=2}^n \sum_{j=2, j \neq i}^n \left\langle \sum_{k=0}^{m(\lambda_i)-1} \alpha_i^k A b_i^k, \sum_{k=0}^{m(\lambda_j)-1} \alpha_j^k b_j^k \right\rangle_B \\
		&\overset{(ii)}{\le}  -b \sum_{i=2}^n \left| \left| \left( \sum_{k=0}^{m(\lambda_i)-1} \alpha_i^k b_i^k \right) \right| \right|_B^2 + 0 \\
		&= -b \sum_{i=2}^n \left\langle \sum_{k=0}^{m(\lambda_i)-1} \alpha_i^k b_i^k, \sum_{k=0}^{m(\lambda_i)-1} \alpha_i^k b_i^k \right\rangle_B \\
		&= -b  \left\langle \sum_{i=2}^n \sum_{k=0}^{m(\lambda_i)-1} \alpha_i^k b_i^k, \sum_{i=2}^n \sum_{k=0}^{m(\lambda_i)-1} \alpha_i^k b_i^k \right\rangle_B \\
		&= - b \left| \left| \sum_{i=2}^n \sum_{k=0}^{m(\lambda_i)-1} \alpha_i^k b_i^k \right| \right|_B^2.
		\end{align*}
		
		As a next step we note that the set $\mathcal{P}(\mathcal{S}) - \bar{m}:= \{x \in \mathbb{R}^S: \exists m \in \mathcal{P}(\mathcal{S}): x = m - \bar{m} \}$ is compact and that $$(\alpha_i^k)_{i \in \{1, \ldots, n\}, k \in \{1, \ldots, m(\lambda_i)-1\}} \mapsto \alpha_1^0 \bar{m} + \sum_{i=2}^n \sum_{k=0}^{m(\lambda_i)-1} \alpha_i^k b_i^k$$ is a homeomorphism.
		Thus, also 
		$$P:= \left\{ (\alpha_i^k)_{i \in \{1, \ldots, n\}, k \in \{1, \ldots, m(\lambda_i)-i\}} \in \mathbb{R}^S: \alpha_1^0 \bar{m} + \sum_{i=2}^n \sum_{k=0}^{m(\lambda_i)-1} \alpha_i^k b_i^k \in \mathcal{P}(\mathcal{S}) - \bar{m}\right\}$$ is compact. 
		
		For $x= \alpha_1^0 \bar{m} + \sum_{i=2}^n \sum_{k=0}^{m(\lambda_i)-1} \alpha_i^k b_i^k \in \mathcal{P}(\mathcal{S})- \bar{m}$ we moreover obtain that
			\begin{align*}
		\langle Ax,x \rangle_B &= \left\langle  A \left(\alpha_1^0 \bar{m} + \sum_{i=2}^n \sum_{k=0}^{m(\lambda_i)-1} \alpha_i^k b_i^k \right), \alpha_1^0 \bar{m} + \sum_{i=2}^n \sum_{k=0}^{m(\lambda_i)-1} \alpha_i^k b_i^k \right\rangle_B \\
		&= (\alpha_1^0)^2 \langle A \bar{m}, \bar{m} \rangle_B + \left \langle  \alpha_1^0  A \bar{m}, \sum_{i=2}^n \sum_{l=0}^{m(\lambda_i)-1}   \alpha_i^k b_i^k\right \rangle_B + \left\langle \sum_{i=2}^n \sum_{k=0}^{m(\lambda_i)-1}  \alpha_i^k  A   b_i^k, \alpha_0^1 \bar{m} \right\rangle_B \\
		& \quad  + \left\langle A \cdot \sum_{i=2}^n \sum_{k=0}^{m(\lambda_i)-1} \alpha_i^k b_i^k, \sum_{i=2}^n \sum_{k=0}^{m(\lambda_i)-1} \alpha_i^k b_i^k \right\rangle_B \\
		&= 0 + 0 + 0  + \left\langle A \cdot \sum_{i=2}^n \sum_{k=0}^{m(\lambda_i)-1} \alpha_i^k b_i^k, \sum_{i=2}^n \sum_{k=0}^{m(\lambda_i)-1} \alpha_i^k b_i^k \right\rangle_B \\
		&\le  -b \cdot  \left| \left| \sum_{i=2}^n \sum_{k=0}^{m(\lambda_i)-1} \alpha_i^k b_i^k \right| \right|_B^2 \\
		&=  -b \cdot \underbrace{\frac{\left| \left| \sum_{i=2}^n \sum_{k=0}^{m(\lambda_i)-1} \alpha_i^k b_i^k \right| \right|_B^2}{\left| \left| \alpha_1^0 \bar{m} + \sum_{i=2}^n \sum_{k=0}^{m(\lambda_i)-1} \alpha_i^k b_i^k \right| \right|_B^2}}_{=:D(\alpha)>0} \cdot \left| \left| \underbrace{\alpha_1^0 \bar{m} + \sum_{i=2}^n \sum_{k=0}^{m(\lambda_i)-1} \alpha_i^k b_i^k}_{=x} \right| \right|_B^2 \\
		& \le - b \underbrace{\left( \min_{\alpha \in P} D(\alpha)\right)}_{=:D>0} ||x||_B^2
		\end{align*}
		As a final preparation we note that, by the same reasoning as in the proof of Theorem \ref{Theorem:Existence}, the set $\mathcal{P}(\mathcal{S})$ is flow invariant for $\dot{m}(t) = f^d(m(t))$.
		Thus, the set $\mathcal{P}(\mathcal{S})- \bar{m}$ is flow invariant for $\dot{x}(t) = \tilde{f}^d(x(t))$.
		
		By definition of the derivative (which in particular yields $||\tilde{f}^d(x)-Ax||_B \in o(||x||_B)$ in a neighbourhood of $0$) and Cauchy's inequality we have
		$$0=\lim_{x \rightarrow 0} \frac{||\tilde{f}^d(x)- Ax||_B}{||x||_B} = \lim_{x \rightarrow 0} \frac{||\tilde{f}^d(x)- Ax||_B\cdot ||x||_B}{||x||_B^2} \ge  \lim_{x \rightarrow 0} \frac{\langle \tilde{f}^d(x)-Ax, x\rangle_B}{||x||_B^2}= 0.$$
		Since for all $x \in \mathcal{P}(\mathcal{S})- \bar{m}$ we have $\langle Ax,x\rangle_B \le -b D||x||_B^2$, there is a $\delta>0$ such that for all $x \in \overline{N_\delta(0)} \cap (\mathcal{P}(\mathcal{S})- \bar{m})$ it holds that $\langle \tilde{f}^d(x),x \rangle_B \le -cD ||x||_B^2.$ 
		If $\tilde{x} \in N_\delta(0)$ such that $\mathcal{D}(\tilde{x} + \bar{m}) \neq \{d\}$, then make $\delta$ smaller such that $\mathcal{D}(x + \bar{m}) = \{d\}$ for all $x \in \mathcal{P}(\mathcal{S})$.
		
		Now let $x_0 \in N_\delta(0) \cap (\mathcal{P}(\mathcal{S}) - \bar{m})$.
		Then by Peano's existence theorem there is a solution $x: [0,\infty) \rightarrow \mathbb{R}^S$ of the initial value problem $\dot{x}(t)= \tilde{f}^d(x(t))$, $x(0)=x_0$ and, furthermore, any solution $x: [0,t_0) \rightarrow \mathbb{R}^S$ can be extended on $[0,\infty)$.
		Let $x: [0,t_0]\rightarrow \mathbb{R}^S$ be a solution curve of the differential equation $\dot{x}(t)= \tilde{f}^d(x(t))$ in $\overline{N_\delta(0)}$ and assume that $x(t) \neq 0$ for all $0 \le t \le t_0$. 
		(If $x(\tilde{t})=0$ for some $\tilde{t}\ge 0$, then $x(t)=0$ for all $t \ge \tilde{t}$.) 
		
		Then it holds that
		\begin{equation}
			\label{Equation:ExponentialConvergenceNonConstant}
			\frac{\partial}{\partial t} ||x(t)||_B  = \frac{1}{||x(t)||_B} \langle \dot{x}(t), x(t) \rangle_B \le -cD ||x(t)||_B,
		\end{equation}
		which means that $||x(t)||_B$ is strictly decreasing on $[0,t_0]$.
		Thus, $x(t) \in N_\delta(0)$ for all $t \in [t_0, t_0+\tilde{\epsilon}]$.
		Repeating this argument, we obtain by \citet[Section 8.5]{HirschNonlinearSink} that $x(t) \in N_\delta(0)$ for all $t \ge 0$.
		Furthermore, the estimate \eqref{Equation:ExponentialConvergenceNonConstant} yields that $||x(t)||_B \le e^{-cDt} ||x(0)||_B$, which is the desired exponential convergence.
	\end{proof}

	These theorems can be directly applied in examples:
	Indeed, we obtain for the consumer choice model introduced in Section \ref{Section:Process} and analysed in \citet{NeumannComputation} that local convergence happens to any deterministic stationary equilibrium where the equilibrium distribution does not equal the boundary value $k_1$ or $k_2$, respectively.
	Also in a simplified version of the corruption model of \citet{KolokoltsovCorruption2017}, which has also been analysed in \citet{NeumannComputation}, we obtain local convergence for several stationary equilibria having a deterministic equilibrium strategy that is unique for the equilibrium point. More precisely, we obtain for any parameter choice local convergence towards those deterministic stationary equilibria where the equilibrium distribution lies in the interior of $\mathcal{P}(\mathcal{S})$  and for some parameter constellations we also obtain local convergence towards the deterministic stationary equilibria where the equilibrium distribution is $(1,0,0)$ or $(0,1,0)$.
		
	\section{Global Convergence for a Two Strategy Setting}
	\label{Section:Global}
	
	The question of global convergence  ``Given an arbitrary initial condition $m_0 \in \mathcal{P}(\mathcal{S})$ does any trajectory converge towards some mean field equilibrium?'' is much more complex. 
	Here we provide a statement for the case that $\mathcal{U}:= \{d \in D^s: \mathcal{D}(m) = \{d\} \}$ consists of exactly two strategies, i.e. $\mathcal{U} = \{d^1,d^2\}$.
	The statement does not directly yield the desired convergence statement, instead we only obtain convergence towards equilibria with a deterministic equilibrium strategy or that the trajectory remains in a set where the two strategies from $\mathcal{U}$ are simultaneously optimal.
	However, relying on example-specific properties, we can then often prove the convergence towards the mixed strategy equilibria by hand.
	
	If $\mathcal{U} = \{d^1,d^2\}$ then the differential inclusion \eqref{Equation:DifferentialInclusion} describing the myopic adjustment process simplifies substantially:
	Define
	$$g(m):=\left( V^{d^2}(m)-V^{d^1}(m) \right) \cdot 1 = \left( (\beta I - Q^{d^2}(m))^{-1} r^{d^2}(m) - (\beta I - Q^{d^1}(m))^{-1} r^{d^1}(m)\right) \cdot 1$$ and assume that $g$ is twice continuously differentiable, i.e. assume that $Q_{ija}$ and $r_{ia}$ are twice continuously differentiable for all $i,j \in \mathcal{S}$ and $a \in \mathcal{A}$ on some open superset $O$ of $\mathcal{P}(\mathcal{S})$.
	Then
	$$F(m):= \begin{cases}
	\left( \sum_{i \in \mathcal{S}} m_i Q_{ij}^{d^1}(m) \right)_{j \in \mathcal{S}} & g(m)<0 \\
	\text{conv} \left\{ 	\left( \sum_{i \in \mathcal{S}}  m_i Q_{ij}^{d^1}(m) \right)_{j \in \mathcal{S}}, 	\left( \sum_{i \in \mathcal{S}}  m_i Q_{ij}^{d^2}(m) \right)_{j \in \mathcal{S}}  \right\} & g(m) =0 \\
		\left( \sum_{i \in \mathcal{S}} m_i Q_{ij}^{d^2}(m) \right)_{j \in \mathcal{S}} &g(m)>0
	\end{cases}.$$
	
	This means, that on $g(m) \neq 0$ the trajectory is the trajectory of a nonlinear Markov chain with generator $Q^{d^1}(\cdot)$ whenever $g(m)<0$ and $Q^{d^2}(\cdot)$ whenever $g(m)>0$. 
	These processes are a generalization of a classical Markov chain with the new feature that the transition probabilities do not only depend on the current state, but also on the current distribution of the process. For more details consider \citet{KolokoltsovNonLinearMakov} and (in particular regarding the long-term behaviour) \citet{NeumannNonlinearMC}.
	Thus, the processes are characterized through the transition probabilities $(P_{ij}(t,m))_{i,j\in \mathcal{S}}$, which describe the probability to be in state $j$ and time $t$ when at time $0$ the state was $i$ and the distribution was $0$, or (non-uniquely) through the marginal distributions $\Phi^t_i(m)$, which describes the probability to be in state $i$ at time $t$ when the initial distribution was $m$. 
	One can show that it is indeed sufficient to characterize a nonlinear Markov chain through a nonlinear generator, that is a Lipschitz continuous function $Q: \mathcal{P}(\mathcal{S}) \rightarrow \mathbb{R}^{S \times S}$ such that $Q(m)$ is a conservative generator for all $m \in \mathcal{P}(\mathcal{S})$.
	
	In the following theorem, we desire that the considered nonlinear Markov chains behave well in the long-term.
	As in the theory of standard Markov chains, the invariant distribution is a central tool for this analysis and it solves the (now non-linear) equation $0 = Q(m) m^T$.
	 However, a weaker condition than classical ergodicity is enough. Indeed, it suffices to require that the nonlinear Markov chain converges in the limit towards some invariant distribution, that is that for all $m_0 \in \mathcal{P}(\mathcal{S})$ there is an invariant distribution $\bar{m}(m_0)$ such that 
	\[
	\lim_{t \rightarrow \infty} || \Phi^t(m_0) - \bar{m}(m_0) || = 0.
	\]
	This condition is indeed weaker than ergodicity, see \citet[Section 4.2]{NeumannNonlinearMC}.
	
	With these preparations, we are able to formulate and prove the global convergence theorem:
	
	\begin{theorem}
		\label{Theorem:GlobalConvergence}
		Assume that for all $m \in O$ such that $g(m)=0$ it holds that $\nabla g(m)\neq 0$.
		Furthermore, assume that the nonlinear Markov chains with transition rate matrix functions $Q^{d^1}(m)$ and $Q^{d^2}(m)$ converge in the limit towards some stationary distribution.
		
		\begin{itemize}
			\item[(i)] If for all $m \in O$ such that $g(m)=0$ it holds that
			$$\langle (Q^{d^1}(m))^T m, \nabla g(m) \rangle >0 \quad \text{and} \quad \langle (Q^{d^2}(m))^T m, \nabla (-g)(m)\rangle <0,$$ then the myopic adjustment processes converges towards some stationary mean field equilibrium with deterministic equilibrium strategy from $\mathcal{U}$.
			\item[(ii)] If for all $m \in O$ such that $g(m)=0$ it holds that
			$$\langle (Q^{d^1}(m))^T m, \nabla g(m) \rangle <0 \quad \text{and} \quad \langle (Q^{d^2}(m))^T m, \nabla (-g)(m) \rangle >0,$$ then the myopic adjustment processes converges towards some stationary mean field equilibrium with deterministic equilibrium strategy from $\mathcal{U}$.
			\item[(iii)] If for all $m \in O$ such that $g(m)=0$ it holds that
			$$\langle (Q^{d^1}(m))^Tm, \nabla g(m)\rangle \ge 0 \quad \text{and} \quad \langle (Q^{d^2}(m))^T m, \nabla (-g)(m)\rangle \ge 0,$$ then the myopic adjustment process either converges towards a deterministic stationary mean field equilibrium with equilibrium strategy from $\mathcal{U}$ or there is a $T>0$ such that the processes satisfies $g(m(t))=0$ for all $t >T$.
		\end{itemize} 
	\end{theorem}

	The gradient conditions of the theorem have an intuition: More precisely, the conditions in case (i) state that when $g(m)=0$ the population distribution heads to the set where the strategy $d^1$ is optimal, and the conditions in case (ii) state that when $g(m)=0$ the population distribution heads to the set where the strategy $d^2$ is optimal. In case (iii) the conditions say, that when $g(m)=0$ and the population chooses $d^1$ the distribution tends into the set where $d^2$ is optimal and when $g(m)=0$ and the population chooses $d^2$ the distribution tends into the set where $d^1$ is optimal.

	\begin{proof}[Proof of Theorem \ref{Theorem:GlobalConvergence}]
		We first note that if there is a $T \ge 0$ such that the trajectory satisfies $g(m(t)) <0$ or $g(m(t))>0$ for all $t \ge T$, then the solution of \eqref{Equation:DifferentialInclusion} is also a solution of $\dot{m}(t) = Q^{d^1}(m(t))$ or $\dot{m}(t) = Q^{d^2}(m(t))$, respectively, which means that $(m(t))_{t \ge T}$ are the marginals of a nonlinear Markov chain.
		By assumption, we thus obtain convergence towards some stationary point, which is, since $g(m(t))<0$ or $g(m(t))>0$ for all $t \in [T,\infty)$, a stationary equilibrium.
		
		In case (i) whenever $g(m(T))=0$ for some $T\ge 0$, then $g(m(t))>0$ for all $t \ge T$.
		Indeed, assume that $g(m(t))=0$ for all $t \in [T,T+\epsilon]$ for some $\epsilon>0$.
		Then for almost all $t \in [T,T+\epsilon]$ it holds that
		$$\frac{\partial}{\partial t} g(m(t)) = \left\langle \lambda Q^{d^1}(m(t))^T m(t) + (1- \lambda) Q^{d^2}(m(t))^T m(t), \nabla g(m(t)) \right\rangle >0,$$ which is a contradiction.
		Similarly, if we assume that $g(m(t))<0$ for all $t \in(T, T+\epsilon)$ with $\epsilon>0$, then it holds for almost all $t \in (T, T+\epsilon)$ that
		$$\frac{\partial}{\partial t} g(m(t)) = \left\langle Q^{d^1}(m(t))^T m(t), \nabla g(m(t)) \right\rangle >0,$$
		again a contradiction.
		Thus, it either holds that $g(m(t))<0$ for all $t \ge 0$ or that $g(m(t))>0$ for all $t \ge T$ with $T\ge 0$, which by the first observation yields the desired convergence.
			
		In case (iii) we have that whenever $g(m(T))=0$ for some $T \ge 0$, then $g(m(t))= 0$ for all $t \ge T$.
		Indeed, assume that there is a $\tilde{T} := \inf \{t \ge T: g(m(t)) \neq 0\} < \infty$.
		If $\tilde{T} = \inf \{ t \ge T: g(m(t))<0\}$, then since $m$ is absolutely continuous and $g$ is continuously differentiable, we obtain for some $\epsilon_1>0$ that
		$$0 > \frac{\partial}{\partial t} g(m(t)) = \langle Q^{d^1}(m(t))^T m(t), \nabla g(m(t)) \ge 0$$ for almost all $t \in [\tilde{T}, \tilde{T} + \epsilon_1]$, a contradiction.
		Similarly, if $\tilde{T} = \inf \{ t \ge T: g(m(t))>0\}$, then we obtain for some $\epsilon_2>0$ that
		$$0 < \frac{\partial}{\partial t} g(m(t)) = \langle Q^{d^2}(m(t))^T m(t)), \nabla g(m(t)) \le 0$$ for all $t \in [\tilde{T}, \tilde{T}+\epsilon_2]$, a contradiction.		
		Thus, either $g(m(t))<0$ for all $t \ge 0$, or $g(m(t))>0$ for all $t \ge 0$, in which case we obtain convergence towards some stationary equilibrium with a deterministic equilibrium strategy, or there is a $T>0$ such that $g(m(t))=0$ for all $t \ge T$.
	\end{proof}
	
	\begin{example}
		\label{Example:Nonlinear}
	Let us consider the following example, which consists of two ``good'' states, where a positive reward is earned, and one ``bad'' state, where no reward is earned. 
	The agents in the ``good'' state face congestion effects, namely there is a risk, increasing in the share of individuals in that state, to go to the ``bad'' state. 
	The control options are to switch between the two good states. 
	One can interpret this model as a stylized version to model the choice between two mobile phone providers, where the customer faces the risk of a breakdown in connection that increases in the share of customers using the same provider. 
	For simplicity, we assume that agents in the ``bad'' state have no choice option, but recover into each of the two states with equal probability.
	
	\begin{figure}[h]
		\label{ExampleConsumerChoiceCongestion}
		\begin{center}
			\begin{tikzpicture}
			
			\node[state] at (0,0) (A)     {1:1P};
			\node[state] at (8,0) (B)     {2:1P};
			\node[state] at (4, -6.5)  (C)     {3:0P};
			
			\draw[every loop, >=latex, auto=right] 
			(A) edge[bend right=20, auto=left] node {\textcolor{blue}{$b$}/\textcolor{red}{$0$}} (B)
			(B) edge[bend right=20] node {\textcolor{blue}{$b$}/\textcolor{red}{$0$}} (A)
			(B) edge[bend right=20] node {$em_2+\epsilon$} (C)
			(C) edge[bend right=20] node {$\lambda$} (B)
			(A) edge[bend right=20] node {$em_1+\epsilon$} (C)
			(C) edge[bend right=20] node {$\lambda$} (A);
			
			\end{tikzpicture}
		\end{center}
		\caption{Representation of Example \ref{Example:Nonlinear}}
	\end{figure}
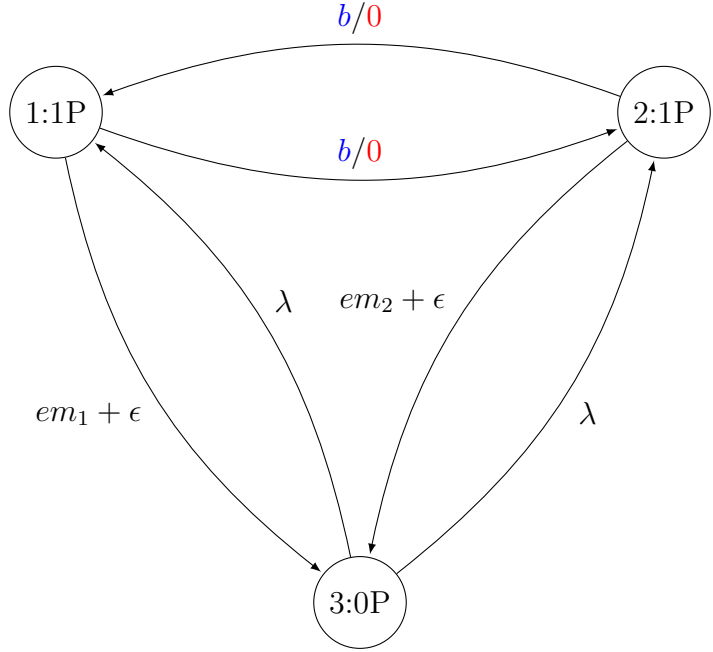 The formal characterization is given by $\mathcal{S}=\{1,2,3\}$ and $\mathcal{A}= \{ \mathit{change}, \mathit{stay}\}$ together with
	\begin{align*}
	Q_{\cdot \cdot \mathit{change}} &= \begin{pmatrix}
	-(b+em_1+\epsilon) & b & em_1 + \epsilon \\
	b & -(b+e m_2+ \epsilon) & em_2 + \epsilon \\
	\lambda & \lambda & -2\lambda
	\end{pmatrix} \\ 
	Q_{\cdot \cdot \mathit{stay}} &= \begin{pmatrix}
	-(em_1 + \epsilon) & 0 & em_1 + \epsilon \\
	0 & -(em_2 + \epsilon) & em_2 + \epsilon \\
	\lambda & \lambda & -2 \lambda
	\end{pmatrix}
	\end{align*} and $r_{\cdot \mathit{stay}}=r_{\cdot \mathit{change}} =(1,1,0)$, where all constants are strictly positive. 
	A visualization of the model is given in Figure \ref{ExampleConsumerChoiceCongestion}.
	
	In \citet{NeumannDissertation} it is shown that there are infinitely many mixed strategy equilibria with equilibrium distribution $$\left( \frac{\sqrt{4 \lambda^2 + 8 \lambda \epsilon + \epsilon^2}-2 \lambda - \epsilon}{2 \epsilon}, \frac{\sqrt{4 \lambda^2 + 8 \lambda \epsilon + \epsilon^2}-2 \lambda - \epsilon}{2 \epsilon}, \frac{4 \epsilon + 4 \lambda - \sqrt{4 \lambda^2 + 8 \epsilon \lambda + \epsilon^2}}{2 \epsilon} \right)$$ with equilibrium strategies satisfying $\pi_{1, \mathit{change}} = \pi_{2, \mathit{change}}$.
	
	To apply Theorem \ref{Theorem:GlobalConvergence} we first note, that in \citet{NeumannNonlinearMC} it is shown that the relevant Markov chains are strongly ergodic and thus, converges to some limit distribution. 
	Moreover, choosing $\mathcal{U} = \{ \{ \mathit{change}\} \times \{\mathit{stay}\}, \{\mathit{stay} \} \times \{ \mathit{change}\}\}$, $O = (- \frac{\epsilon}{b}, \infty) \times (-\frac{\epsilon}{b}, \infty) \times \mathbb{R}$ and $g(m)= m_1 -m_2$ we obtain that
	$$\left\langle Q^{cs}(m))^T m, \nabla -g(m) \right \rangle =  2 bm_2 \ge 0 \quad \text{and} \quad \left\langle Q^{sc}(m))^Tm, \nabla g(m) \right\rangle = 2 bm_2 \ge 0.$$
	
	Thus,  Theorem \ref{Theorem:GlobalConvergence} yields that either convergence towards a stationary equilibrium with an equilibrium strategy from $\mathcal{U}$ happens or that there is a $T\ge 0$ such that $g(m(t))=0$ for all $t \ge T$. 
	Since there is no stationary equilibrium with a equilibrium strategy from $\mathcal{U}$ it is clear that there is a $T\ge 0$ such that $g(m(t))=0$ for all $t \ge T$, which means that $m_1(t)=m_2(t)$ for all $t \ge T$.
	Thus, also $\dot{m}_1(t)=\dot{m}_2(t)$.
	By \eqref{Equation:DifferentialInclusion}, this yields that
	\begin{align*}
	&- \pi_{1, \mathit{change}}(t) b m_1(t) - e m_1(t)^2 - \epsilon m_1(t) + \pi_{2,\mathit{change}}(t) b m_1(t) + \lambda m_3(t) \\
	&\quad = \pi_{1,\mathit{change}}(t) b m_1(t) - \pi_{2, \mathit{change}}(t) b m_1(t) - em_1(t)^2 - \epsilon m_1(t) + \lambda m_3(t), \\
	&\text{ i.e. } - \pi_{1, \mathit{change}}(t) b m_1(t) + \pi_{2,\mathit{change}}(t) bm_1(t)=0.
	\end{align*} 
	Thus, for almost all $t \ge T$ the trajectory of the myopic adjustment process has to satisfy $$\dot{m}_1(t)  = -e m_1(t)^2 - \left( \epsilon + 2 \lambda \right) m_1(t) + \lambda,$$ which is a Riccati equation, for which $[0,1]$ is flow invariant and for which a unique classical solution for any initial condition $m_0 \in [0,1]$ exists.
	Numerical simulations indicate that in our setting with initial conditions $m_0 \in [0,1]$ convergence towards the distribution of the stationary mixed strategy equilibria is likely. \demo
	\end{example}

	\appendix
	\section{Appendix}
	\label{appendix}
	
	\begin{proof}[Proof of Theorem \ref{Theorem:Existence}]
		We show that the conditions of Lemma 5.1 in \citet{DeimlingMultivaued} are satisfied, as this yields the desired existence statement. 
		More precisely, we show in the following that
		\begin{itemize}
			\item[(i)] $F$ is upper semicontinuous,
			\item[(ii)] $F(m)$ is a closed, convex set for all $m \in \mathcal{P}(\mathcal{S})$,
			\item[(iii)] there is a constant $c>0$ such that $||F(m)||:= \sup \{ ||y|| : y \in F(m) \} \le c (1 + ||m||)$ for all $m \in \mathcal{P}(\mathcal{S})$, and
			\item[(iv)] $F(m) \cap T_{\mathcal{P}(\mathcal{S})} (m) \neq \emptyset$ for all $M \in \mathcal{P}(\mathcal{S})$ where 
			\begin{align*}
			T_{\mathcal{P}(\mathcal{S})}(m) &= \left\{ y \in \mathbb{R}^S: \liminf_{h \downarrow 0} \frac{d(m+hy, \mathcal{P}(\mathcal{S}))}{h} = 0 \right\} \\
			&= \left\{ y \in \mathbb{R}^S : y_i \ge 0 \forall i \in \mathcal{S} \text{ s.t. } m_i = 0 \wedge \sum_{i \in \mathcal{S}} y_i = 0 \right\}
			\end{align*} \citep[Proposition 5.1.7]{AubinDifferentialInclusions}.			
		\end{itemize}
		
		(i): Let $m \in \mathcal{P}(\mathcal{S})$ and let $N \subseteq \mathbb{R}^S$ be an open set such that $F(m) \subseteq N$. 
		Since $2^{\mathbb{R}^S}$ equipped with the Hausdorff distance $H(\cdot, \cdot)$ is a metric space, it suffices to consider sets of the form $N=N_\epsilon(F(m))$ with $\epsilon>0$.
		Since $\mathcal{D}(m) = \{d \in D^s: V^d(m) = V^\ast(m)\}$, we find for any $d' \in D^s \setminus \mathcal{D}(m)$ a constant $c_{d'}>0$ such that $V^{d'}(m)< V^\ast - c_{d'}$.
		Therefore by finiteness of $D^s$, the constant $c= \min_{d \in D^s \setminus \mathcal{D}(m)} c_{d'} >0$ satisfies $V^{d'}(m) < V^\ast(m) - c$ for all $d' \in D^s \setminus \mathcal{D}(m)$.
		Since $V^d: \mathcal{P}(\mathcal{S}) \rightarrow \mathbb{R}^S$ is continuous for every $d \in D^s$, there is a $\delta_d>0$ such  that $\tilde{m} \in N_{\delta_d}(m) \Rightarrow ||V^d(\tilde{m})-V^d(m)|| < \frac{c}{3}$.
		In particular, choosing $\delta_1 := \min_{d \in D^d} \delta_d$ we obtain for all $\tilde{m} \in N_{\delta_1}(m)$ and all $d' \in D^s \setminus \mathcal{D}(m)$ that pointwise
		\begin{align*}
			V^\ast(\tilde{m}) - V^{d'}(\tilde{m}) &\ge \left( V^\ast(\tilde{m}) - V^\ast(m) \right) + \left( V^\ast(m) - V^{d'}(m) \right) + \left( V^{d'}(m) - V^{d'}(\tilde{m} \right) \\
			> - \frac{c}{3} + c - \frac{c}{3} > 0.
		\end{align*} Thus, $d' \notin \mathcal{D}(\tilde{m})$, that is $\mathcal{D}(\tilde{m}) \subseteq \mathcal{D}(m)$.
		Furthermore, for $d \in D^s$ the map $F^d: \mathcal{P}(\mathcal{S})\rightarrow \mathbb{R}^S$, $m \mapsto \left( \sum_{i \in \mathcal{S}} \sum_{a \in \mathcal{A}} m_i Q_{ija}(m) d_{ia} \right)_{j \in \mathcal{S}}$ is continuous.
		Therefore, there is a $\delta_{2,d}>0$ such that $\tilde{m} \in N_{\delta_{2,d}}(m) \Rightarrow ||F^d(\tilde{m}) - F^d(m)|| < \epsilon$. 
		Set $\delta_2 = \min_{d \in D^s} \delta_{2,d}$
		Then for $\delta:= \min \{\delta_1, \delta_2\}$ it holds that $\tilde{m} \in N_\delta(m) \Rightarrow F(\tilde{m}) \subseteq N_\epsilon(F(m))$.
		
		(ii) Since for all $m \in \mathcal{P}(\mathcal{S})$ the set $F(m)$ is convex polytope, it is closed and convex.
		
		(iii) Since $Q_{ija}(\cdot)$ is Lipschitz continuous for all $i,j \in \mathcal{S}$ and all $a \in \mathcal{A}$ it is moreover uniformly bounded in $m \in \mathcal{P}(\mathcal{S})$, $i,j \in \mathcal{S}$ and $a \in \mathcal{A}$ by some constant, which we denote by $M$.
		Thus, for all $d \in D^s \supseteq \mathcal{D}(m)$ we have
		$$||F^d(m)|| = \sum_{j \in \mathcal{S}} \left|\sum_{i \in \mathcal{S}} \sum_{a \in \mathcal{A}} m_i Q_{ija}(m) d_{ia} \right| \le \sum_{j \in \mathcal{S}} \sum_{i \in \mathcal{S}} \sum_{a \in \mathcal{A}} m_i M d_{ia} = SM,$$ which implies, since $F(m)$ is a convex hull of $(F^d(m))_{d \in \mathcal{D}(m)}$ that $||F(m)||_1 \le SM$.
		
		(iv) The condition is trivially satisfied for all $m \in\text{int} ( \mathcal{P}(\mathcal{S}))$ since then $T_{\mathcal{P}(\mathcal{S})}(m) = \mathbb{R}^S$ and $F(m) \neq \emptyset$ because $\mathcal{D}(m) \neq \emptyset$.
		Now, let $m \in \partial \mathcal{P}(\mathcal{S})$ be a boundary point. 
		Then there is at least one $j \in \mathcal{S}$ such that $m_j=0$.  
		Since the only non-positive column entry of $Q_{\cdot ja}(m)$ is in row $j$, this implies that $F^d(m)_j \ge 0$. 
		Moreover,
		$$\sum_{j \in \mathcal{S}} F^d(m)_j = \sum_{i \in \mathcal{S}} \sum_{ \in \mathcal{A}} \underbrace{\left(\sum_{j \in \mathcal{S}} Q_{ija}(m) \right)}_{=0} d_{ia}  = 0,$$
		which yields that $F^d(m) \in T_{\mathcal{P}(\mathcal{S})}(m)$.
		Since $F(m)$ is a convex combination of $F^d(m)$ and $T_{\mathcal{P}(\mathcal{S})}(m)$ is convex, the desired claim follows.
	\end{proof}
	
	\bibliographystyle{plainnat}
	\bibliography{literatureMyopic}

\end{document}